\RequirePackage{fix-cm}
\documentclass[smallextended]{svjour3}       
\smartqed  
\usepackage{graphicx,bm,url,amssymb,amsmath,hyperref,color}
\def\no{\noindent}
\def\pmatrix{\left(\begin{array}}
\def\endpmatrix{\end{array}\right)}

\def\CC{\mathbb{C}}

\def\RR{\mathbb{R}}

\def\I{{\cal I}}

\def\P{{\cal P}}

\def\dd{\mathrm{d}}

\def\ii{\mathrm{i}}

\def\bfb{{\bm{b}}}
\def\bfc{{\bm{c}}}

\def\bfzero{{\bm{0}}}

\def\bfgamma{{\bm{\gamma}}}

\def\bfeta{{\bm{\eta}}}

\def\bfphi{{\bm{\phi}}}

\def\eps{\varepsilon}

\def\aa{\alpha}

\def\cpr{$^\copyright\,$}
\def\tol{\mbox{\em tol}}

\begin{document}

\title{Numerical solution of FDE-IVPs by using Fractional HBVMs: the {\tt fhbvm} code}


\titlerunning{Numerical solution of FDE-IVPs by using FHBVMs}        

\author{Luigi Brugnano   \and  Gianmarco Gurioli \and Felice Iavernaro}


\institute{L.\,Brugnano \at
             Dipartimento di Matematica e Informatica ``U.\,Dini''\\
             Universit\`a di Firenze,  Italy\\ 
              \email{luigi.brugnano@unifi.it}  \\
              \url{https://orcid.org/0000-0002-6290-4107}
                        \and
           G.\,Gurioli \at
             Dipartimento di Matematica e Informatica ``U.\,Dini''\\
             Universit\`a di Firenze,  Italy\\ 
              \email{gianmarco.gurioli@unifi.it}  \\
              \url{https://orcid.org/0000-0003-0922-8119}
                        \and
           F.\,Iavernaro \at
           Dipartimento di Matematica\\ 
           Universit\`a di Bari Aldo Moro, Italy\\
            \email{felice.iavernaro@uniba.it}\\
            \url{https://orcid.org/0000-0002-9716-7370}
}

\date{Received: date / Accepted: date}

\maketitle

\begin{abstract} In this paper we describe the efficient numerical implementation of {\em Fractional HBVMs}, a class of methods recently introduced for solving systems of fractional differential equations. The reported arguments are implemented in the Matlab\cpr code {\tt fhbvm}, which is made available on the web.  An extensive experimentation of the code is reported, to give evidence of its effectiveness.

\end{abstract}

\section{Introduction} 

Fractional differential equations have become a common description tool across a variety of applications (see, e.g.,  the classical references \cite{Di2010,Po1999} for an introduction). For this reason, their numerical solution has been the subject of many researches (see, e.g. \cite{AMN2014,DFF2002,DiFoFr2005,LYC16,Lu1985,SOG17,ZK2014}), with the development of corresponding software (see, e.g., \cite{CCP2023,Ga2015,Ga2018}). In this context, the present contribution is addressed for solving  {\em initial value problems for fractional differential equations (FDE-IVPs)} in the form 
\begin{equation}\label{ivp}
y^{(\aa)}(t) = f(y(t)), \qquad t\in[0,T],\qquad y(0) = y_0\in\RR^m,
\end{equation}
where, for the sake of brevity, we have omitted the argument $t$ for $f$. 
Here, for ~$\aa\in(0,1)$, ~$y^{(\aa)}(t) \equiv D^\aa y(t)$~ is the Caputo fractional derivative:\footnote{As is usual, $\Gamma$ denotes the Euler gamma function, such that, for $x>0$, $x\Gamma(x)=\Gamma(x+1)$.}
\begin{equation}\label{Dalfa}
D^\aa g(t) = \frac{1}{\Gamma(1-\aa)} \int_0^t (t-x)^{-\aa} \left[\frac{\dd}{\dd x}g(x)\right]\dd x,
\end{equation}
The Riemann-Liouville integral associated to (\ref{Dalfa}) is given by:
\begin{equation}\label{Ialfa}
I^\aa g(t) = \frac{1}{\Gamma(\aa)}\int_0^t (t-x)^{\aa-1} g(x)\dd x.
\end{equation}
Consequently, the solution of (\ref{ivp}) can be formally written as:
\begin{equation}\label{sol0}
y(t) = y_0 + I^{\aa} f(y(t)) \equiv y_0 +  \frac{1}{\Gamma(\aa)}\int_0^t (t-x)^{\aa-1} f(y(x))\dd x, \qquad t\in[0,T].
\end{equation}

The numerical method we shall consider, relies on {\em Fractional HBVMs (FHBVMs)}, a class of methods recently introduced in \cite{BBBI2023}, as an extension of  Hamiltonian Boundary Value Methods (HBVMs), special low-rank Runge-Kutta methods originally devised for Hamiltonian problems (see, e.g., \cite{BI2016,BI2018}), and later extended along several directions (see, e.g., \cite{ABI2019,ABI2022-1,ABI2023,BBBI2023,BFCIV2022,BI2022}), including the numerical solution of FDEs. A main feature of HBVMs is the fact that they can gain spectrally accuracy, when approximating ODE-IVPs \cite{ABI2020,BMIR2019,BMR2019}, and such a feature has been recently extended to the FDE case \cite{BBBI2023}.

With this premise, the structure of the paper is as follows: in Section~\ref{start} we recall the main facts about the numerical solution of FDE-IVPs proposed in \cite{BBBI2023}; in Section~\ref{details} we provide full implementation details for the Matlab\cpr code {\tt fhbvm} used in the numerical tests; in Section~\ref{num} we report an extensive experimentation of the code, providing some comparisons with another existing one; at last, a few conclusions are given in Section~\ref{fine}.

\section{Fractional HBVMs}\label{start}
To begin with, in order to obtain a piecewise approximation to the solution of the problem,  we consider a partition of the integration interval in the form:
\begin{equation}\label{tn}
t_n = t_{n-1} + h_n,  \qquad n=1,\dots,N, 
\end{equation}
where \,$h_n>0$, \,$n=1,\dots,N$,\, and such that
\begin{equation}\label{T}
t_0=0, \qquad t_N= T\equiv \sum_{n=1}^N h_n.
\end{equation} 
Further, by setting
\begin{equation}\label{yn}
y_n(ch_n) \,:=\, y(t_{n-1}+ch_n), \qquad c\in[0,1],\qquad n=1,\dots,N,
\end{equation}
the restriction of the solution of (\ref{ivp}) on the interval $[t_{n-1},t_n]$, and taking into account (\ref{sol0}) and (\ref{tn})--(\ref{T}), one obtains, for $t\,\equiv\, t_{n-1}+ch_n$, $c\in[0,1]$,
\begin{eqnarray}
\nonumber \lefteqn{
y(t) ~\equiv~y_n(ch_n) ~=~ y_0 + \frac{1}{\Gamma(\aa)}\int_0^{t_{n-1}+ch_n} (t_{n-1}+ch_n-x)^{\aa-1} f(y(x))\dd x}\\
\nonumber
&=&~y_0 + \frac{1}{\Gamma(\aa)}\sum_{\nu=1}^{n-1} \int_{t_{\nu-1}}^{t_\nu} (t_{n-1}+ch_n-x)^{\aa-1}f(y(x))\dd x\\[2mm]  
\nonumber
&&\qquad + ~\frac{1}{\Gamma(\aa)} \int_{t_{n-1}}^{t_{n-1}+ch_n} (t_{n-1}+ch_n-x)^{\aa-1}f(y(x))\dd x\\[2mm]  
\nonumber
&=&~y_0 + \frac{1}{\Gamma(\aa)}\sum_{\nu=1}^{n-1} \int_0^{h_\nu} (t_{n-1}-t_{\nu-1}+ch_n-x)^{\aa-1}f(y_\nu(x))\dd x\\[2mm]  
\nonumber
&&\qquad + ~\frac{1}{\Gamma(\aa)} \int_0^{ch_n} (ch_n-x)^{\aa-1}f(y_n(x))\dd x\\[2mm]  
\nonumber
&=&~y_0 + \frac{1}{\Gamma(\aa)}\sum_{\nu=1}^{n-1} 
h_\nu^\aa \int_0^1 \left( \frac{t_{n-1}-t_{\nu-1}}{h_\nu}+c\frac{h_n}{h_\nu}-\tau \right)^{\aa-1}f(y_\nu(\tau h_\nu))\dd \tau\\[2mm]  
\nonumber
&&\qquad + ~\frac{h_n^\aa}{\Gamma(\aa)} \int_0^c (c-\tau)^{\aa-1}f(y_n(\tau h_n))\dd \tau\\[2mm]  
\nonumber
&=&~y_0 + \frac{1}{\Gamma(\aa)}\sum_{\nu=1}^{n-1} 
h_\nu^\aa \int_0^1 \left( \sum_{i=\nu}^{n-1} \frac{h_i}{h_\nu}+c\frac{h_n}{h_\nu}-\tau \right)^{\aa-1}f(y_\nu(\tau h_\nu))\dd \tau\\[2mm]  \label{ynch}
&&\qquad + ~\frac{h_n^\aa}{\Gamma(\aa)} \int_0^c (c-\tau)^{\aa-1}f(y_n(\tau h_n))\dd \tau, \qquad c\in[0,1].
\end{eqnarray}
To make manageable the handling of the ratios $h_i/h_\nu$, $i=\nu,\dots,n$, $\nu=1,\dots,n-1$, we shall hereafter consider the following choices for the mesh (\ref{tn})--(\ref{T}):

\begin{description}

\item[\bf - Graded mesh.] In order to cope with possible singularities in the derivative of the vector field at the origin, we consider the  {\em graded mesh}
\begin{equation}\label{hnr}
h_n = r h_{n-1} \equiv h_n = r^{n-1} h_1, \qquad n=1\dots,N,
\end{equation}
where $r>1$ and $h_1>0$ satisfy, by virtue of (\ref{tn})--(\ref{T}),
\begin{equation}\label{rT}
h_1\frac{r^N-1}{r-1} = T.
\end{equation}
As a result, one obtains that (\ref{ynch}) becomes:
\begin{eqnarray}\nonumber
\lefteqn{y_n(ch_n)~=~y_0~ +}\\[1mm] \nonumber
&& \frac{h_1^\aa}{\Gamma(\aa)}\sum_{\nu=1}^{n-1} r^{\aa(\nu-1)}
\int_0^1 \left( \frac{r^{n-\nu}-1}{r-1}+cr^{n-\nu}-\tau \right)^{\aa-1}f(y_\nu(\tau r^{\nu-1}h_1))\dd \tau\\[2mm]  \nonumber
&&\qquad + ~\frac{h_1^\aa r^{\aa(n-1)}}{\Gamma(\aa)} \int_0^c (c-\tau)^{\aa-1}f(y_n(\tau r^{n-1}h_1))\dd \tau\\ \nonumber
&=:& ~\phi_{n-1}^\aa(c;h_1,r,y)  \,+\, \frac{h_1^\aa r^{\aa(n-1)}}{\Gamma(\aa)} \int_0^c (c-\tau)^{\aa-1}f(y_n(\tau r^{n-1}h_1))\dd \tau,\\
&&\qquad\qquad c\in[0,1].\label{grad}
\end{eqnarray}

\bigskip
\item[\bf - Uniform mesh.] This case is equivalent to allowing $r=1$ in (\ref{hnr}). Consequently, from (\ref{tn})--(\ref{T}) we derive 
\begin{equation}\label{unif}
h_n \equiv h_1 := \frac{T}N, \quad n=1,\dots,N, \quad \Rightarrow\quad t_n = nh_1, \quad n=0,\dots,N.
\end{equation}
As a result,  (\ref{ynch}) reads:
\begin{eqnarray}\nonumber
\lefteqn{y_n(ch_n)~\equiv~y_n(ch_1)}\\ \nonumber
&=&y_0 + \frac{h_1^\aa}{\Gamma(\aa)}\sum_{\nu=1}^{n-1} 
\int_0^1 \left( n-\nu+c-\tau \right)^{\aa-1}f(y_\nu(\tau h_1))\dd \tau\\[2mm]  \nonumber
&&\qquad + ~\frac{h_1^\aa}{\Gamma(\aa)} \int_0^c (c-\tau)^{\aa-1}f(y_n(\tau h_1))\dd \tau,\\ \nonumber
&\equiv& ~\phi_{n-1}^\aa(c;h_1,1,y)  \,+\, \frac{h_1^\aa}{\Gamma(\aa)} \int_0^c (c-\tau)^{\aa-1}f(y_n(\tau h_1))\dd \tau  \\
&&\qquad\qquad c\in[0,1].\label{cost}
\end{eqnarray}

\end{description}

\begin{remark}
As is clear, in order to obtain an accurate approximation of the solution, it is important to establish which kind of mesh (graded or uniform) is appropriate. Besides this, also a proper choice of the parameters $h_1$, $r$, and $N$ in (\ref{hnr})--(\ref{rT}) is crucial. Both aspects will be studied in Section~\ref{graduni}.
\end{remark} 

\subsection{Quasi-polynomial approximation}\label{sigma}

We now discuss a piecewise quasi-polynomial approximation to the solution of (\ref{ivp}),
$$\sigma(t)\approx y(t), \qquad t\in[0,T],$$
such that
\begin{equation}\label{sign}
\sigma_n(ch_n) \,:=\, \sigma(t_{n-1}+ch_n), \qquad c\in[0,1],\qquad n=1,\dots,N,
\end{equation}
is the approximation to $y_n(ch_n)$, defined in (\ref{yn}). According to \cite{BBBI2023}, such approximation 
will be derived through the following steps:
\begin{enumerate}
\item expansion of the vector field, in each sub-interval $[t_{n-1},t_n]$, $n=1,\dots,N$, (recall (\ref{tn})--(\ref{T})) along a suitable orthonormal polynomial basis;

\medskip
\item truncation of the infinite expansion, in order to obtain a local polynomial approximation.

\end{enumerate}

Let us first consider the expansion of the vector field along the orthonormal polynomial basis, w.r.t. the weight function
\begin{equation}\label{om}
\omega(x) = \aa(1-x)^{\aa-1}, \qquad s.t.\qquad \int_0^1 \omega(x)\,\dd x=1,
\end{equation}
resulting into a scaled and shifted family of Jacobi polynomials:\footnote{Here, $\overline P_j^{\,(a,b)}(x)$ denotes the $j$-th Jacobi polynomial with parameters $a$ and $b$, in $[-1,1]$.}
$$ 
P_j(x) := \sqrt{\frac{2j+\aa}\aa} \,\overline{P}_j^{\,(\aa-1,0)}(2x-1), \qquad x\in[0,1], \qquad j=0,1,\dots,
$$ 
such that
\begin{equation}\label{orton}
\int_0^1 \omega(x) P_i(x)P_j(x)\dd x = \delta_{ij}, \qquad i,j=0,1,\dots.
\end{equation}
In so doing, for $n=1,\dots,N$, one obtains:
\begin{equation}\label{expf}
f(y_n(ch_n)) = \sum_{j\ge0} P_j(c)\gamma_j(y_n), \qquad c\in[0,1],
\end{equation}
with (see (\ref{om}))
\begin{equation}\label{gammaj}
\gamma_j(y_n) = \int_0^1\omega(\tau) P_j(\tau)f(y_n(\tau h_n))\dd\tau, \qquad j=0,1,\dots.
\end{equation}
Consequently, the FDE (\ref{ivp}), can be rewritten as:
\begin{equation}\label{ivpn}
y_n^{(\aa)}(ch_n) =  \sum_{j\ge0} P_j(c)\gamma_j(y_n), \quad c\in[0,1],\quad n=1,\dots,N, \qquad y_1(0) = y_0. 
\end{equation}

The approximation is derived by truncating the infinite series in (\ref{expf}) to a finite sum with $s$ terms, thus replacing (\ref{ivpn}) with a series of local problems, whose vector field is a polynomial of degree $s$:
\begin{equation}\label{ivpn_s}
\sigma_n^{(\aa)}(ch_n) =  \sum_{j=0}^{s-1} P_j(c)\gamma_j(\sigma_n), \quad c\in[0,1],\quad n=1,\dots,N, \qquad 
\sigma_1(0) = y_0,
\end{equation}
with $\gamma_j(\sigma_n)$ defined similarly as in (\ref{gammaj}):
\begin{equation}\label{gammaj_s}
\gamma_j(\sigma_n) = \int_0^1\omega(\tau)P_j(\tau)f(\sigma_n(\tau h_n))\dd\tau, \qquad j=0,\dots,s-1.
\end{equation}
As a consequence, (\ref{grad}) will be approximated as:
\begin{eqnarray}\nonumber
\lefteqn{\sigma_n(ch_n)~=}\\ \nonumber
&& y_0\,+\,\frac{h_1^\aa}{\Gamma(\aa)}\sum_{\nu=1}^{n-1} r^{\aa(\nu-1)}
\int_0^1 \left( \frac{r^{n-\nu}-1}{r-1}+cr^{n-\nu}-\tau \right)^{\aa-1}\,\sum_{j=0}^{s-1} P_j(\tau)\gamma_j(\sigma_\nu)\dd \tau\\[2mm]  \nonumber
&&\qquad + ~\frac{h_1^\aa r^{\aa(n-1)}}{\Gamma(\aa)} \int_0^c (c-\tau)^{\aa-1}\sum_{j=0}^{s-1} P_j(\tau)\gamma_j(\sigma_n)\dd\tau~=~y_0\,+\\ \nonumber
&&\,h_1^\aa\sum_{\nu=1}^{n-1} r^{\aa(\nu-1)}\sum_{j=0}^{s-1}\left[\frac{1}{\Gamma(\aa)}
\int_0^1 \left( \frac{r^{n-\nu}-1}{r-1}+cr^{n-\nu}-\tau \right)^{\aa-1} P_j(\tau)\dd \tau\right]\gamma_j(\sigma_\nu)\\[2mm]  \nonumber
&&\qquad + ~h_1^\aa r^{\aa(n-1)} \sum_{j=0}^{s-1}\left[\frac{1}{\Gamma(\aa)}\int_0^c (c-\tau)^{\aa-1} P_j(\tau)\dd\tau\right]\gamma_j(\sigma_n)\\ \nonumber
&&=:~y_0\,+\,h_1^\aa\sum_{\nu=1}^{n-1} r^{\aa(\nu-1)}\sum_{j=0}^{s-1} J_j^\aa
\left( \frac{r^{n-\nu}-1}{r-1}+cr^{n-\nu}\right)\gamma_j(\sigma_\nu)\\ \nonumber
&&\qquad +~h_1^\aa r^{\aa(n-1)} \sum_{j=0}^{s-1}I^\aa P_j(c)\,\gamma_j(\sigma_n)\\ 
&&=:~\phi_{n-1}^{\aa,s}(c;h_1,r,\sigma)  ~+~   h_1^\aa r^{\aa(n-1)}\sum_{j=0}^{s-1} I^\aa P_j(c)\,\gamma_j(\sigma_n),
\qquad c\in[0,1],\label{grad_s}
\end{eqnarray}
where (see (\ref{Ialfa}))
\begin{equation}\label{Ijalfa}
I^\aa P_j(c) = \frac{1}{\Gamma(\alpha)}\int_0^c (c-\tau)^{\aa-1}P_j(\tau)\dd\tau,\qquad j=0,\dots,s-1,
\end{equation}
is the Riemann-Liouville integral of $P_j(c)$, and, for $x\ge1$:
\begin{equation}\label{Jjalfa}
J_j^\aa(x) = \frac{1}{\Gamma(\aa)}\int_0^1 (x-\tau)^{\aa-1}P_j(\tau)\dd\tau, \qquad j=0,\dots,s-1.
\end{equation}
The efficient numerical evaluation of the integrals (\ref{Ijalfa}) and (\ref{Jjalfa}) will be explained in detail in Section~\ref{fractint}.
\begin{remark}
We observe that, for $c=x=1$, by virtue of (\ref{om})--(\ref{orton}) one has:
\begin{equation}\label{ceq1}
I^\aa P_j(1) = J_j^\aa(1) = \frac{\delta_{j0}}{\Gamma(\aa+1)}, \qquad j=0,\dots,s-1.
\end{equation}
\end{remark}

Similarly, when using a uniform mesh, by means of similar steps as above, (\ref{cost}) turns out to be approximated by:
\begin{eqnarray}\nonumber
\lefteqn{\sigma_n(ch_n)~=}\\[1mm] \nonumber
&& y_0\,+\,h_1^\aa\sum_{\nu=1}^{n-1} \sum_{j=0}^{s-1} J_j^\aa
\left( n-\nu+c\right)\gamma_j(\sigma_\nu)
\,+\, h_1^\aa \sum_{j=0}^{s-1}I^\aa P_j(c)\,\gamma_j(\sigma_n)\\ 
&\equiv& ~\phi_{n-1}^{\aa,s}(c;h_1,1,\sigma)  ~+~   h_1^\aa \sum_{j=0}^{s-1} I^\aa P_j(c)\,\gamma_j(\sigma_n),
\qquad c\in[0,1].\qquad\label{cost_s}
\end{eqnarray}

In both cases, the approximation at $t_n$ is obtained, by setting $c=1$ and taking into account (\ref{ceq1}), as:
\begin{eqnarray}\nonumber
\sigma_n(h_n) &=& \phi_{n-1}^{\aa,s}(1;h_1,r,\sigma) \,+\, \frac{h_1^\aa r^{\aa(n-1)}}{\Gamma(\aa+1)}\gamma_0(\sigma_n)\\
&\equiv& \phi_{n-1}^{\aa,s}(1;h_1,r,\sigma) \,+\, \frac{h_n^\aa}{\Gamma(\aa+1)}\gamma_0(\sigma_n),\label{ynew}
\end{eqnarray}
which formally holds also in the case of a uniform mesh (see (\ref{unif})) by setting $r=1$.
\bigskip
\bigskip

\subsection{The fully discrete method}\label{fhbvms}

Quoting Dahlquist and Bj\"ork \cite{DB2008}, {\em ``as is well known, even many relatively simple integrals cannot be expressed in finite terms of elementary functions, and thus must be evaluated by numerical methods''}. In our framework, this obvious statement means that, in order to obtain a numerical method, at step $n$ the Fourier coefficients $\gamma_j(\sigma_n)$ in (\ref{gammaj_s}) need to be approximated by means of a suitable quadrature formula. 
Fortunately enough, this can be done up to machine precision by using a Gauss-Jacobi formula of order $2k$ based at the zeros of $P_k(c)$, $c_1,\dots,c_k$, with corresponding weights (see (\ref{om}))
$$b_i = \int_0^1 \omega(c)\ell_i(c)\dd c, \qquad \ell_i(c) = \prod_{j\ne i}\frac{c-c_j}{c_i-c_j}, \qquad i=1,\dots,k,$$ 
by choosing  a value of $k$, $k\ge s$, large enough. In other words,
\begin{equation}\label{gammajn}
\gamma_j^n \,:=\, \sum_{i=1}^k b_i P_j(c_i) f(\sigma_n(c_ih_n)) \,\doteq\, \gamma_j(\sigma_n), \qquad j=0,\dots,s-1,
\end{equation}
where $\,\doteq\,$ means {\em equal within machine precision}. Because of this, and for sake of brevity, we shall continue using $\sigma_n$ to denote the fully discrete approximation (compare with (\ref{grad_s}), or with (\ref{cost_s}), in case $r=1$):\footnote{Hereafter, we shall use $h_n$ in place of $h_1r^{n-1}$.}
\begin{equation}\label{sigman_d}
\sigma_n(ch_n) = \phi_{n-1}^{\aa,s}(c;h_1,r,\sigma) + h_n^\aa\sum_{j=0}^{s-1} I^\aa P_j(c)\gamma_j^n, \qquad c\in[0,1].
\end{equation}
As is clear, the coefficientes $\gamma_j^\nu$, $j=0,\dots,s-1$, $\nu=1,\dots,n-1$, needed to evaluate $\phi_{n-1}^{\aa,s}(c;h_1,r,\sigma)$, have been already computed at the previous time-steps.

We observe that, from (\ref{gammajn}), in order to compute the (discrete) Fourier coefficients, it is enough evaluating (\ref{sigman_d}) only at the quadrature abscissae $c_1,\dots,c_k$. In so doing, by combining (\ref{gammajn}) and (\ref{sigman_d}), one obtains a discrete problem in the form:
\begin{eqnarray}\nonumber
\gamma_j^n &=& \sum_{i=1}^k b_i P_j(c_i) f\left(\phi_{n-1}^{\aa,s}(c_i;h_1,r,\sigma) + h_n^\aa\sum_{\ell=0}^{s-1} I^\aa P_\ell(c_i)\gamma_\ell^n\right), \\
&& j~=~0,\dots,s-1.\label{dispro}
\end{eqnarray}
Once it has been solved, according to (\ref{ynew}), the approximation of $y(t_n)$ is given by:
\begin{equation}\label{ynew_d}
\bar y_n \,:=\, \sigma_n(h_n) \,=\,\phi_{n-1}^{\aa,s}(1;h_1,r,\sigma) \,+\, \frac{h_n^\aa}{\Gamma(\aa+1)}\gamma_0^n.
\end{equation}

It is worth noticing that the discrete problem (\ref{dispro}) can be cast in vector form, by introducing the (block) vectors
$$\bfgamma^n = \pmatrix{c} \gamma_0^n\\ \vdots \\ \gamma_{s-1}^n\endpmatrix\in\RR^{sm},\qquad \bfphi_{n-1}^{\aa,s} = \pmatrix{c} \phi_{n-1}^{\aa,s}(c_1;h_1,r,\sigma)\\ \vdots\\ \phi_{n-1}^{\aa,s}(c_k;h_1,r,\sigma)\endpmatrix\in\RR^{km},$$
and the matrices
$$
\P_s = \pmatrix{ccc} P_0(c_1) & \dots & P_{s-1}(c_1)\\ \vdots & &\vdots\\ P_0(c_k) & \dots & P_{s-1}(c_k)\endpmatrix,
\quad \I_s^\aa = \pmatrix{ccc} I^\aa P_0(c_1)& \dots &I^\aa P_{s-1}(c_1)\\ \vdots & &\vdots\\ I^\aa P_0(c_k) & \dots &I^\aa P_{s-1}(c_k)\endpmatrix\quad \in\RR^{k\times s},$$
$$\Omega = \pmatrix{ccc} b_1\\ &\ddots\\ &&b_k\endpmatrix\in\RR^{k\times k},$$
as:
\begin{equation}\label{vform}
\bfgamma^n = \P_s^\top\Omega \otimes I_m f\left( \bfphi_{n-1}^{\aa,s} +h_n^\aa\I_s^\aa\otimes I_m\bfgamma^n\right),
\end{equation}
with the obvious notation for the function $f$, evaluated in a vector of (block) dimension $k$, of denoting the (block) vector of dimension $k$ containing the function $f$ evaluated in each (block) entry. As observed in \cite{BBBI2023}, the (block) vector
\begin{equation}\label{Y}
Y \,:=\, \bfphi_{n-1}^{\aa,s} +h_n^\aa\I_s^\aa\otimes I_m\bfgamma^n\in\RR^{km},
\end{equation}
appearing at the r.h.s. in (\ref{vform}) as argument of $f$, satisfies the equation 
\begin{equation}\label{vformk}
Y ~=~ \bfphi_{n-1}^{\aa,s} +h_n^\aa\I_s^\aa\P_s^\top\Omega \otimes I_m\, f(Y),
\end{equation}
obtained combining (\ref{vform}) and (\ref{Y}). 
Consequently, it can be regarded as the stage vector of a Runge-Kutta type method with Butcher tableau
\begin{equation}\label{fhbvmks}
\begin{array}{c|c} \bfc & \I^\aa_s \P_s^\top\Omega \\ \hline  \\[-3mm] & \bfb^\top \end{array}, \qquad
\bfb=\pmatrix{ccc} b_1,&\dots,&b_k\endpmatrix^\top, \qquad \bfc=\pmatrix{ccc} c_1,&\dots,&c_k\endpmatrix^\top.
\end{equation}

\begin{remark}
Though the two formulations (\ref{vform}) and (\ref{vformk}) are equivalent each other, nevertheless, the former has (block) dimension $s$ {\em independently} of the considered value of $k$ ($k\ge s$), which is the (block) dimension of the latter. Consequently, in the practical implementation of the method, the discrete problem (\ref{vform}) is the one to be preferred, since it is independent of the number of stages.
\end{remark}

When $\aa=1$, the Runge-Kutta method (\ref{fhbvmks}) reduces to a HBVM$(k,s)$ method \cite{ABI2020,ABI2022,BFCIV2022,BI2016,BI2018,BI2022,BMIR2019,BMR2019}. Consequently, we give the following definition \cite{BBBI2023}.

\begin{definition}
The method defined by  (\ref{fhbvmks}) (i.e., (\ref{ynew_d})--(\ref{vform})) is called {\em fractional HBVM with parameters $(k,s)$}. In short, FHBVM$(k,s)$.
\end{definition}

The efficient numerical solution of the discrete problem (\ref{vform}) will be considered in Section~\ref{blend}.

\section{Implementation issues}\label{details}

In this section we report all the implementation details used in the Matlab\cpr code {\tt fhbvm}, which will be used in the numerical tests reported in Section~\ref{num}.

\subsection{Graded or uniform mesh?}\label{graduni}

The first relevant problem to face is the choice between a graded or a uniform mesh and, in the former case, also choosing appropriate values for the parameters $r$, $N$, and $h_1$ in (\ref{hnr})--(\ref{rT}).  We start considering a proper choice of this latter parameter, i.e., $h_1$ which, in turn, will allow us to choose the type of mesh, too. For this purpose, the user is required to provide, in input, a convenient integer value $M>1$, such that,  if a uniform mesh is appropriate, then the stepsize is given by: 
\begin{equation}\label{hmax}
h = \frac{T}M.
\end{equation}
It is to be noted that, since the code is using a method with spectral accuracy in time, the value of $M$ should be as small as possible. That said, we set $h_1^0=h$ and apply the code on the interval $[0,h_1^0]$ both in a single step, and by using a graded mesh of 2 sub-intervals defined by a value of $r:=\hat r\equiv 3$.  As a result, the two sub-intervals, obtained by solving (\ref{rT}) for $h_1$, with
$$ r=3, \qquad N=2, \qquad T=h_1^0,$$
are\,\footnote{It is to be noticed that the division by 4 is done without introducing round-off errors.}
$$[0,h_1^0/4], \qquad [h_1^0/4,h_1^0].$$
In so doing, we obtain two approximations to $y(h_1^0)$, say $y_1$ and $y_2$. According to the analysis in \cite{BBBI2023}, if\,\footnote{Here,  $./$ means the componentwise division between two vectors, and $|y_2|$ denotes the vector with the absolute values of the entries of $y_2$.} 
\begin{equation}\label{erro}
\| (y_1-y_2)./(1+|y_2|)\|_\infty \le \tol,
\end{equation}
with \tol\, a suitably small tolerance,\footnote{In view of the spectral accuracy of the method, this tolerance is only slightly larger than the machine epsilon.} then this means the stepsize $h_1:=h_1^0$ is appropriate. Moreover, in such a case, a uniform mesh with $N\equiv M$ turns out to be appropriate, too. 

Conversely, the procedure is repeated on the sub-interval $[0,h_1^1] \equiv [0,h_1^0/4]$, and so forth, until a convenient initial stepsize $h_1$ is obtained. This process can be repeated up to a suitable maximum number of times that, considering that at each iteration the current guess of $h_1$ is divided by 4, allows for a substantial reduction of the initial value (\ref{hmax}). Clearly, in such a case, a graded mesh is needed. At the end of this procedure, we have then chosen the initial stepsize $h_1$ which, if the procedure ends at the $\ell$-th iteration, for a convenient $\ell>1$, is given by:
\begin{equation}\label{h1ell}
h_1 = 4^{1-\ell} h \equiv 4^{1-\ell} \frac{T}M.
\end{equation}
We need now to appropriately choose the parameters $r$ and $N$ in (\ref{hnr})--(\ref{rT}).   To simplify this choice,
we require that the last stepsize in the mesh be equal to (\ref{hmax}). Consequently, we would ideally satisfy the following requirements:
$$h_1\frac{r^N-1}{r-1} = T, \qquad h_1r^{N-1} = \frac{T}M,$$
which, by virtue of (\ref{h1ell}), become
$$4^{1-\ell}\frac{r^N-1}{r-1}=M, \qquad 4^{1-\ell}r^{N-1}=1.$$
Combining the two equations then gives:
\begin{equation}\label{rN0}
r = \frac{M-4^{1-\ell}}{M-1}>1, \qquad  N = 1 +\log_r(4^{\ell-1}).
\end{equation}
As is clear, this value of $N$ is not an integer, in general, so that we shall choose, instead:
\begin{equation}\label{N}
N = \lceil 1 +\log_r(4^{\ell-1})\rceil.
\end{equation}
\begin{remark}
From (\ref{N}), and considering that $Nh_1<T$ (conversely, a uniform mesh could have been used) and, by virtue of (\ref{h1ell}), one has:
\begin{equation}\label{bound}
2\le N < 4^{\ell-1}M.
\end{equation}
\end{remark}
As is clear, using the value of $N$ in (\ref{N}) in place of that in (\ref{rN0}) implies that we need to recompute $r$, so that the requirement (now $h_1$, $N$, and $T$ are given)
\begin{equation}\label{erre}
h_1\frac{r^N-1}{r-1}=T 
\end{equation}
is again fulfilled. Equation (\ref{erre}) can be rewritten as
\begin{equation}\label{psi}
r = \left[ 1 + (r-1)\beta \right]^{1/N} =: \psi(r), \qquad \beta := \frac{T}{h_1}>1,
\end{equation}
thus inducing the iterative procedure
\begin{equation}\label{iter}
r_0>1 \mbox{\quad(given)}, \qquad r_{i+1} = \psi(r_i), \qquad i=0,1,\dots.
\end{equation}
As a convenient choice for $r_0$, one can use the guess for $r$ defined in (\ref{rN0}).
The following result can be proved.

\begin{theorem}\label{converge}
There exists a unique $\bar r>1$ satisfying $\bar r=\psi(\bar r)$, and the iteration (\ref{iter}) globally converges to this value over the interval $(1, +\infty)$.
\end{theorem}
\begin{proof} 
A direct computation shows that:
\begin{eqnarray}
\psi(1)=1,  \qquad   \psi'(r)>0,  ~\textrm{for~} r>1, \label{psi_cond1} \\
\psi'(1)=\frac{\beta}{N}>1,~ \lim_{r\rightarrow +\infty} \psi'(r)=0,  \qquad  \psi''(r)<0, ~\textrm{for~} r>1. \label{psi_cond2}
\end{eqnarray}
From (\ref{psi_cond1}) we deduce that the (positive) mapping $\psi(r)$ is strictly increasing and admits $(1,+\infty)$ as invariant set, since  $r>1$ implies $\psi(r)>\psi(1)=1$. 

From (\ref{psi_cond2}) we additionally deduce that $\psi(r)$ is concave and $\psi(r)>r$ for $r\in(1,1+\eps)$, with $\eps>0$ sufficiently small.  These properties and the fact that  $\psi'(r) \rightarrow 0$, for $r\rightarrow+\infty $,   imply that: 
\begin{itemize}
\item[-] the equation $r=\psi(r)$  admits a unique solution ~$\bar r>1$ ~and~ $\psi'(\bar r)<1$;
\medskip
\item[-] $\psi\big((1,\bar r)\big) \subset (1,\bar r)$  ~and~ $\psi\big((\bar r,+\infty)\big) \subset (\bar r,+\infty)$ ~(since $\psi$ is increasing);
\medskip 
\item[-] $\psi(r)>r$ ~for ~$1<r<\bar r$ ~and~ $\psi(r)<r$ ~for~ $r>\bar r$. 
\end{itemize}
From the two latter properties we conclude that, for any $r_0>1$, the sequence generated by (\ref{iter}) converges monotonically to $\bar r$. \qed
\end{proof}\bigskip

Consequently, we have derived the parameters $h_1$ in (\ref{h1ell}), $N$ in (\ref{N}), and $r$ satisfying (\ref{erre}) of the graded mesh (\ref{hnr})--(\ref{rT}), the latter one obtained  by a few iterations of (\ref{iter}).

\begin{remark}
In the actual implementation of the code, we allow the use of a uniform mesh also when $\ell=2$ steps of the previous procedure are required, provided that $M$ has a moderate value (say, $M\le5$). Consequently, for the final mesh, $h_1=h/4$, $r=1$, and $N=4M$.
\end{remark}

\subsection{Approximating the fractional integrals}\label{fractint}
The practical implementation of the method requires the evaluation of the following integrals (recall (\ref{dispro}) and the definitions  (\ref{Ijalfa})--(\ref{Jjalfa})):
\begin{equation}\label{Ijalfa1}
I^\aa P_j(c_i),\qquad i=1,\dots,k,\quad j=0,\dots,s-1,
\end{equation}
and
\begin{equation}\label{Jjalfa1}
J_j^\aa\left(\frac{r^\nu-1}{r-1}+c_ir^\nu\right), \qquad i=1,\dots,k,\quad j=0,\dots,s-1, \quad \nu=1\dots,N-1,
\end{equation}
in case a graded mesh is used, or
\begin{equation}\label{Jjalfa2}
J_j^\aa(\nu+c_i), \qquad i=1,\dots,k,\quad j=0,\dots,s-1, \quad \nu=1\dots,N-1,
\end{equation}
in case a uniform mesh is used.

It must be emphasized that all such integrals ((\ref{Ijalfa1}) and (\ref{Jjalfa1}), or (\ref{Ijalfa1}) and (\ref{Jjalfa2})), can be pre-computed once for all, for later use. For their computation, in the first version of the software, we adapted an algorithm based on \cite{ABI2022} which, however, required a quadruple precision, for the considered values of $k$ and $s$. In this respect, the use of the standard {\tt vpa} of Matlab\cpr, which is based on a symbolic computation, turned out to be too slow. For this reason, hereafter we describe two new algorithms for computing the above integrals, which result  to be quite satisfactory, when using the standard double precision IEEE. Needless to say that, since {\tt vpa} is no more required, they turn out to be much faster than the previous ones.

Let us describe, at first, the computation of (\ref{Ijalfa1}). One has, by considering that $k\ge s$ and $c_i\in(0,1)$:
\begin{eqnarray*}
\lefteqn{I^\aa P_j(c_i) ~=~ \frac{1}{\Gamma(\aa)}\int_0^{c_i}(c_i-\tau)^{\aa-1}P_j(\tau)\dd\tau} \\[1mm]
&=&\,\frac{c_i^{\aa-1}}{\Gamma(\aa)}\int_0^{c_i}\left(1-\frac{\tau}{c_i}\right)^{\aa-1}P_j(\tau)\dd\tau\,=\,\frac{c_i^\aa}{\Gamma(\aa)}\int_0^1\left(1-\xi\right)^{\aa-1}P_j(\xi c_i)\dd\xi\\[1mm]
&=&\,\frac{c_i^\aa}{\Gamma(\aa+1)}\int_0^1\aa\left(1-\xi\right)^{\aa-1}P_j(\xi c_i)\dd\xi \,\equiv\,
\frac{c_i^\aa}{\Gamma(\aa+1)}\sum_{\ell=1}^k b_\ell P_j(c_ic_\ell),\\[1mm]
&&\qquad\qquad j=0,\dots,s-1,
\end{eqnarray*}
where the last equality holds because the Jacobi quadrature formula has order $2k$. 
Clearly, for each $i=1,\dots,k$, all the above integrals can be computed in vector form in ``one shot'', by using the usual three-term recurrence to compute the Jacobi polynomials.

Concerning the integrals (\ref{Jjalfa1})-(\ref{Jjalfa2}), let us now consider the evaluation of a generic $J_j^\aa(x)$, $j=0,\dots,s-1$, for $x>1$. In this respect, there is numerical evidence that, for $x\ge1.1$, a high-order Gauss-Legendre formula is able to approximate the required integral to full machine precision. Since we will use a value of $s=20$, we consider, for this purpose, a Gauss-Legendre formula of order $60$, which turns out to be fully accurate. Instead, for $x\in(1,1.1)$, one has: 
\begin{eqnarray*}
\lefteqn{J_j^\aa(x) ~=~\frac{1}{\Gamma(\aa)}\int_0^1 (x-\tau)^{\aa-1}P_j(\tau)\dd \tau}\\[1mm]
&=& \frac{1}{\Gamma(\aa)}\left[\int_0^x (x-\tau)^{\aa-1}P_j(\tau)\dd \tau ~-~ \int_1^x(x-\tau)^{\aa-1}P_j(\tau)\dd \tau
\right]\\[1mm]
&=& \frac{x^{\aa-1}}{\Gamma(\aa)}\int_0^x \left(1-\frac{\tau}x\right)^{\aa-1}P_j(\tau)\dd \tau \\[1mm]
&&  \qquad ~-~ 
 \frac{(x-1)^{\aa-1}}{\Gamma(\aa)}\int_0^{x-1}\left(1-\frac{c}{x-1}\right)^{\aa-1}P_j(1+c)\dd c
\\[1mm]
&=& \frac{x^\aa}{\Gamma(\aa)}\int_0^1 (1-\xi)^{\aa-1}P_j(\xi x)\dd \xi - \frac{(x-1)^\aa}{\Gamma(\aa)}\int_0^1 (1-\xi)^{\aa-1}P_j(1+\xi(x-1))\dd \xi\\[1mm]
&=& \frac{x^\aa}{\Gamma(\aa+1)}\int_0^1 \aa(1-\xi)^{\aa-1}P_j(\xi x)\dd \xi \\[1mm]
&&\qquad~-~ \frac{(x-1)^\aa}{\Gamma(\aa+1)}\int_0^1 \aa(1-\xi)^{\aa-1}P_j(1+\xi(x-1))\dd \xi\\[1mm]
&\equiv& \frac{x^\aa}{\Gamma(\aa+1)}\sum_{\ell=1}^k b_\ell P_j(c_\ell x)~-~
\frac{(x-1)^\aa}{\Gamma(\aa+1)}\sum_{\ell=1}^k b_\ell P_j(1+c_\ell(x-1)),\\[1mm]
&&\qquad\qquad j=0,\dots,s-1,
\end{eqnarray*}
again, due to the fact that the quadrature is exact for polynomials of degree at most $s-1$. Also in this case, for each fixed $x>1$, all the integrals can be computed in ``one shot'' by using the three-term recurrence of the Jacobi polynomials. 

We observe that the previous expression is exact also for $x=1$, since $J_j^\aa(1)=\delta_{j0}/\Gamma(\aa+1)$.

\subsection{Error estimation}\label{errest}
An estimate of the global error can be derived by computing the solution on a doubled mesh. In other words, if (see (\ref{ynew_d}))
$$\bar y_n\,\approx\, y(t_n), \qquad n=0,\dots,N,$$
with $t_n$ as in (\ref{tn})--(\ref{T}), are the obtained approximations, then
$$e_n \,:=\, y(t_n)-\bar y_n \,\approx\, \hat y_{2n}-\bar y_n, \qquad n=0,\dots,N,$$
where $\hat y_n\approx y(\hat t_n)$, $n=0,\dots,2N$, is the solution computed on a doubled mesh. 

When a uniform mesh (\ref{unif}) is used, the doubled mesh is simply given by:
$$\hat t_n = n\frac{h_1}2\equiv n\frac{T}{2N}, \qquad n=0,\dots,2N.$$

Conversely, when a graded mesh (\ref{hnr})--(\ref{rT}) is considered, the doubled mesh is given by:
$$\hat t_n = \hat t_{n-1} + \hat h_n, \qquad \hat h_n = \hat r \hat h_{n-1}\equiv \hat r^{n-1}\hat h_1, 
\qquad n=1,\dots,2N,$$with $\hat t_0=0$, and  $$\hat r = \sqrt r, \qquad \hat h_1 = h_1\frac{\hat r-1}{r-1}.$$
The choice of $\hat h_1$ is done for having
$$\hat h_1\frac{\hat r^{2N}-1}{\hat r-1} = \hat h_1\frac{r^N-1}{\hat r-1} \equiv h_1\frac{r^N-1}{r-1}=T,$$
according to (\ref{rT}).

\subsection{The nonlinear iteration}\label{blend}

At last, we describe the efficient numerical solution of the discrete problem (\ref{vform}), which has to be solved at the $n$-th integration step. As is clear, the very formulation of the problem induces a straightforward fixed-point iteration:
\begin{equation}\label{vformell}
\bfgamma^{n,\ell} = \P_s^\top\Omega \otimes I_m f\left( \bfphi_{n-1}^{\aa,s} +h_n^\aa\I_s^\aa\otimes I_m\bfgamma^{n,\ell-1}\right),\qquad n=1,2,\dots,
\end{equation}
which can be conveniently started from $\bfgamma^{n,0}=\bf0$. The following straightforward result holds true.

\begin{theorem}\label{fpit}
Assume $f$ be Lipchitz with constant $L$ in in the interval $[t_{n-1},t_n]$. Then, the iteration
(\ref{vformell}) is convergent for all timesteps $h_n$ such that
$$h_n^\aa L\|\P_s^\top\Omega\|\|\I_s^\aa\| < 1.$$
\end{theorem}
\begin{proof} See \cite[Theorem\,2]{BBBI2023}.\qed\end{proof}\bigskip

Nevertheless, as is easily seen, also the simple equation
$$y^{(\aa)} = -\lambda y, \qquad t\in[0,T], \qquad y(0)=1, \qquad T,\lambda\gg1,$$
whose solution is almost everywhere close to 0, after an initial transient, suffers from stepsize limitations, if the fixed-point iteration (\ref{vformell}) is used, since it has to be everywhere proportional to $\lambda^{-1/\aa}$.

In order to overcome this drawback, a Newton-type iteration is therefore needed. Hereafter, we consider the so-called {\em blended iteration} which has been at first studied in a series of papers \cite{B2000,BM2002,BM2007,BM2009}.
It has been implemented in the Fortran codes {\tt BIM} \cite{BM2004}, for ODE-IVPs, and {\tt BIMD} \cite{BMM2006}, for ODE-IVPs and linearly implicit DAEs, and in the Matlab code {\tt hbvm} \cite{BI2016,BIT2011}, for solving Hamiltonian problems. We here consider its adaption for solving (\ref{vform}). By neglecting, for sake of brevity, the time-step index $n$, we then want to solve the equation:
\begin{equation}\label{Ggam}
G(\bfgamma) :=  \bfgamma - \P_s^\top\Omega \otimes I_m f\left( \bfphi^{\aa,s} +h^\aa\I_s^\aa\otimes I_m\bfgamma\right) = \bfzero.
\end{equation}
By setting $f_0'$ the Jacobian of $f$ evaluated at the first entry of $\bfphi^{\aa,s}$, $I=I_s\otimes I_m$, and
\begin{equation}\label{Xs}
X_s^\aa := \P_s^\top\Omega \I_s^\aa,
\end{equation}
the application of the simplified Newton method then reads:
\begin{eqnarray}\label{simpnewt}
\mathrm{solve:}&& \left( I - h^\aa X_s^\aa\otimes f_0' \right)\Delta\bfgamma^\ell = -G(\bfgamma^\ell) \equiv \bfeta^\ell, \\[2mm] \nonumber
\mathrm{set:}&& \bfgamma^{\ell+1}=\bfgamma^\ell+\Delta\bfgamma^\ell,\qquad \ell=0,1,\dots.
\end{eqnarray}
Even though this iteration has the advantage of using a coefficient matrix which is constant at each time-step, nevertheless, its dimension may be large, when either $s$ or $m$ are large. To study a different iteration, able to get rid of this problem, let us decouple the linear system into the various eigenspaces of $f_0'$, thus studying the simpler problem
\begin{eqnarray*}
\mathrm{solve:}&& \left( I_s - h^\aa\mu X_s^\aa \right)\Delta\gamma^\ell = -g(\gamma^\ell)\equiv \eta^\ell, \\[2mm] \nonumber
\mathrm{set:}&& \gamma^{\ell+1}=\gamma^\ell+\Delta\gamma^\ell,\qquad \ell=0,1,\dots,
\end{eqnarray*}
with all involved vectors of dimension $s$, and $\mu\in\sigma(f_0')$ a generic eigenvalue of $f_0'$, and an obvious meaning of $g(\gamma^\ell)$. By setting $q=h^\aa\mu$, the iteration then reads:
\begin{eqnarray}\label{simpnewtmu}
\mathrm{solve:}&& \left( I_s - q X_s^\aa \right)\Delta\gamma^\ell = \eta^\ell, \\[2mm] \nonumber
\mathrm{set:}&& \gamma^{\ell+1}=\gamma^\ell+\Delta\gamma^\ell,\qquad \ell=0,1,\dots.
\end{eqnarray}
Hereafter, we consider the iterative solution of the linear system in (\ref{simpnewtmu}). 
A linear analysis of convergence (in the case the r.h.s. is constant) is then made, as at first suggested in \cite{HS1997,HS1997-1,HSS1996}, and later refined in \cite{BM2002,BM2009}. Consequently, skipping the iteration index $\ell$, let us consider the linear system to be solved:
$$ \left( I_s - q X_s^\aa \right)\Delta\gamma = \eta,$$
and its equivalent formulation, derived considering that matrix (\ref{Xs}) is nonsingular, and
with $\xi>0$ a parameter to be later specified, 
$$ \xi\left( (X_s^\aa)^{-1} - q I_s \right)\Delta\gamma = \xi(X_s^\aa)^{-1}\eta =: \eta_1.$$
Further, we consider the {\em blending} of the previous two equivalent formulations with weights $\theta(q)$ and $I_s-\theta(q)$, where, by setting $O\in\RR^{s\times s}$ the zero matrix,
$$\theta(q) := I_s(1-\xi q)^{-1}\left\{\begin{array}{cc} \approx I_s, &\quad q\approx 0,\\[1mm] \rightarrow O, &q\rightarrow \infty,\end{array}\right.$$
In so doing, one obtains the linear system
\begin{equation}\label{etaq}
M(q)\Delta\gamma = \eta_1+\theta(q)(\eta-\eta_1) =: \eta(q),
\end{equation}
with the coefficient matrix,
$$M(q) = \xi\left( (X_s^\aa)^{-1} - q I_s \right) + \theta(q)\left[\left( I_s - q X_s^\aa \right) -\xi\left( (X_s^\aa)^{-1} - q I_s \right)\right],$$ 
such that \cite{BM2002}: 
$$M(q) \approx\left\{\begin{array}{cc} I_s, &q\approx 0,\\[2mm] -\xi qI_s,& |q|\gg 1.\end{array}\right.$$
This naturally induces the splitting matrix 
\begin{equation}\label{Nq}
N(q) := I_s(1-\xi q)\equiv \theta(q)^{-1}, 
\end{equation}
defining the {\em blended iteration}
\begin{equation}\label{blendi}
\Delta\gamma_i = \left[I_s-\theta(q)M(q)\right]\Delta\gamma_{i-1} +\theta(q)\eta(q), \qquad i=1,2,\dots.
\end{equation}
This latter iteration converges iff the spectral radius of the iteration matrix,
\begin{equation}\label{rom1}
\rho\left(I_s-\theta(q)M(q)\right)=: \rho(q) <1.
\end{equation}
The iteration is said to be $A$-convergent if (\ref{rom1}) holds true for all $q\in\CC^-$, the left-half complex plane, and $L$-convergent if, in addition, $\rho(q)\rightarrow0$, as $q\rightarrow\infty$. Since \cite{BM2002}
$$\theta(0)M(0)=I_s, \qquad \theta(q)M(q)\rightarrow I_s, \quad q\rightarrow\infty,$$
the blended iteration is $L$-convergent iff it is $A$-convergent. For this purpose, we shall look for a suitable choice of the positive parameter $\xi>0$. Considering that $\theta(q)$ is well defined for all $q\in\CC^-$, the following statement easily follows from the maximum modulus theorem.

\begin{theorem}\label{rostar} The blended iteration is $L$-convergent iff the {\em maximum amplification factor},
$$\rho^* := \max_{x>0} \rho(\ii x) \le 1.$$
\end{theorem}

The following result also holds true.

\begin{theorem}\label{eigs} The eigenvalues of the iteration matrix $I_s-\theta(q)M(q)$ are given by:
$$\frac{q(\lambda-\xi)^2}{\lambda(1-q\xi)^2}, \qquad \lambda\in\sigma(X_s^\aa).$$
Consequently, the  maximum amplification factor is given by:
\begin{equation}\label{ros1}
\rho^* = \max_{\lambda\in\sigma(X_s^\aa)} \frac{|\lambda-\xi|^2}{2\xi|\lambda|}.
\end{equation}
\end{theorem}
\begin{proof} See \cite[Theorem\,2 and Equation\,(25)]{BM2002}, by considering that
$$\max_{q\in\CC^-} \frac{q(\lambda-\xi)^2}{\lambda(1-q\xi)^2} ~\equiv~ \max_{x>0}\frac{x|\lambda-\xi|^2}{|\lambda|(1+x^2\xi^2)}$$
is obtained at $x=\xi^{-1}$, so that (\ref{ros1}) follows.\,\qed\end{proof}

\bigskip
Slightly generalizing the arguments in \cite{BM2002}, we then consider the following choice of the parameter $\xi$,
\begin{equation}\label{xi}
\xi = \mathrm{argmin}_{\mu\in\sigma(X_s^\aa)} \max_{\lambda\in\sigma(X_s^\aa)} \frac{|\lambda-|\mu||^2}{2|\mu||\lambda|},
\end{equation}
which is computed once forall, and always provides, in our experiments, an $L$-convergent iteration. In particular, the code {\tt fhbvm} uses, at the moment, $k=22$ and $s=20$: the corresponding maximum amplification factor (\ref{ros1}) is depicted in Figure~\ref{ros1-alfa}, w.r.t. the order $\alpha$ of the fractional derivative, thus confirming this.

\begin{figure}[t]
\centering
\includegraphics[width=9cm]{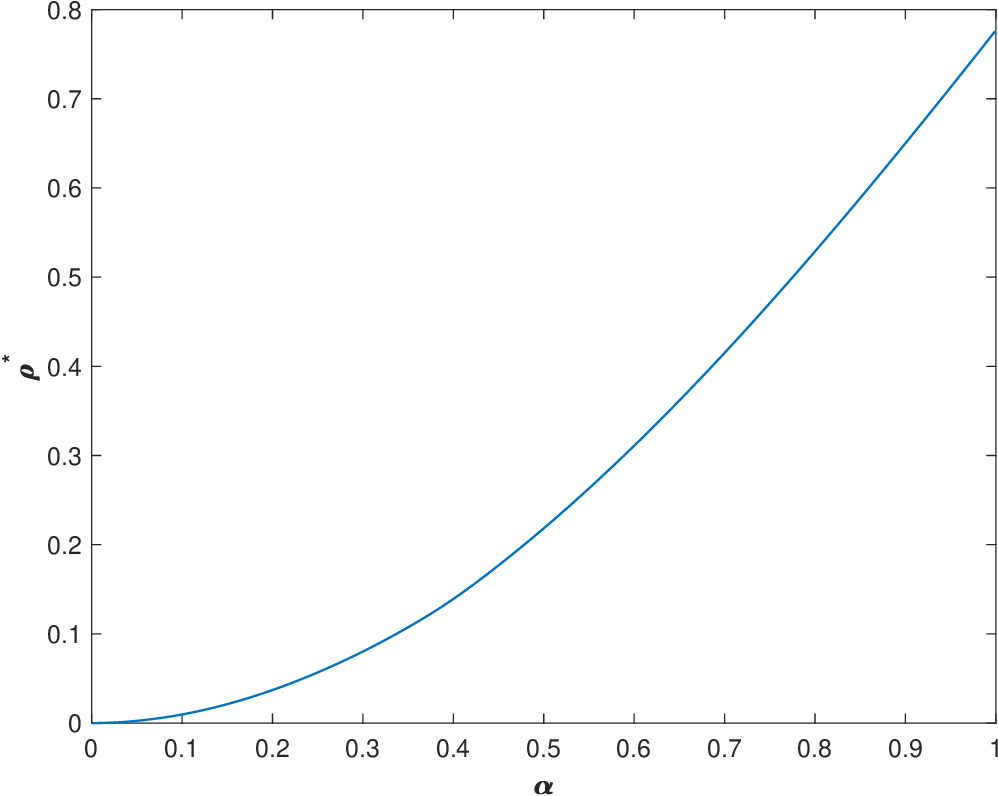}
\caption{Maximum amplification factor (\ref{ros1})--(\ref{xi}), $k=22$ and $s=20$.}
\label{ros1-alfa}
\end{figure}

Coming back to the original problem (\ref{simpnewt}), starting from the initial guess $\Delta\bfgamma=\bfzero$, and updating the r.h.s. as soon as a new approximation to $\bfgamma$ is available, one has that the iteration (\ref{etaq})--(\ref{blendi}) simplifies to:

\begin{eqnarray}\nonumber
\bfeta^\ell &=& -G(\bfgamma^\ell)\\[1mm] \label{blendfin}
\bfeta_1^\ell &=& \xi (X_s^\aa)^{-1}\otimes I_m\, \bfeta^\ell\\[1mm] \nonumber
\bfgamma^{\ell+1} &=& \bfgamma^\ell +I_s\otimes \Theta\left[ \bfeta_1^\ell + I_s\otimes \Theta\left( \bfeta^\ell-\bfeta_1^\ell\right)\right], \qquad \ell=0,1,\dots,
\end{eqnarray}
with $\xi$ chosen according to (\ref{xi}), and
$$
\Theta= \left(I_m- h^\aa\xi f_0'\right)^{-1}.
$$
Consequently, only the factorization of one matrix, having the same dimension $m$ of the problem, is needed. Moreover, the initial guess $\bfgamma^0=\bfzero$ can be conveniently considered in (\ref{blendfin}).

\begin{remark} It is worth mentioning that, due to the properties of the Kronecker product, the iteration (\ref{blendfin}) can be compactly cast in matrix form, thus avoiding an explicit use of the Kronecker product. This implementation has been considered in the code {\tt fhbvm} used in the numerical tests.
\end{remark}

Actually, according to Theorem\,\ref{fpit},  in the code {\tt fhbvm} we automatically switch between the fixed-point iteration (\ref{vformell}) or the blended iteration (\ref{blendfin}), depending on the fact that
$$h^\aa \|f_0'\|\|\P_s^\top\Omega\|\|\I_s^\aa\|\le tol,$$
with $tol<1$ a suitable tolerance.

\section{Numerical Tests}\label{num}

In this section we report a few numerical tests using the Matlab\cpr code {\tt fhbvm}: the code implements a FHBVM(22,20) method using all the strategies discussed in the previous section. The calling sequence of the code is:

\smallskip
\centerline{\tt [t,y,stats,err] = fhbvm( fun, y0, T, M )}
\smallskip

\no where:
\begin{description}
\item{In input:}
\smallskip
\begin{itemize}
\item {\tt fun} is the identifier (or the function handling) of the function evaluating the r.h.s. of the equation (also in vector mode), its Jacobian, and the order $\aa$ of the fractional derivative (see {\tt help fhbvm} for more details);
\item {\tt y0} is the initial condition;
\item {\tt T} is the final integration time;
\item {\tt M} is the parameter in (\ref{hmax}) (it should be as small as possible);
\end{itemize}

\newpage
\medskip
\item{In output:}
\smallskip
\begin{itemize}
\item {\tt t,y} contain the computed mesh and solution;
\item {\tt stats} (optional) is a vector containing the following time statistics:
\begin{enumerate}
\item the pre-processing time for computing the parameters $h_1$, $r$, and $N$ (see Section~\ref{graduni}) and the fractional itegrals (\ref{Ijalfa1}), and (\ref{Jjalfa1}) or (\ref{Jjalfa2});
\item the time for solving the problem;
\item the pre-processing time for computing the fractional itegrals (\ref{Ijalfa1}), and (\ref{Jjalfa1}) or (\ref{Jjalfa2}) for the error estimation;
\item the time for solving the problem on the doubled mesh, for the error estimation;
\end{enumerate}
\item {\tt err} (optional), if specified, contains the estimate of the absolute error. This estimate, obtained on a doubled mesh, is relatively costly: for this reason, when the parameter is not specified, the solution on the doubled mesh is not computed.
\end{itemize}
\end{description}

For the first two problems, we shall also make a comparison with the Matlab\cpr code {\tt flmm2} \cite{Ga2015},\footnote{In particular, the BDF2 method is selected ({\tt method=3}), with the parameters {\tt tol=1e-15} and {\tt itmax=1000}.} in order to emphasize the potentialities of the new code. All numerical tests have been done on a M2-Silicon based computer with 16GB of shared memory, using Matlab\cpr R2023b.

The comparisons will be done by using a so called {\em Work Precision Diagram (WPD)}, where the execution time (in {\tt sec}) is plotted against accuracy. The accuracy, in turn, is measured through the {\em mixed error significant computed digits} ({\tt mescd}) \cite{testset}, defined, by using the same notation seen in (\ref{erro}), as\,\footnote{This definition corresponds to set {\tt atol=rtol} in the definition used in \cite{testset}.}
$${\tt mescd} := -\log_{10} \max_{i=0,\dots,N} \| (y(t_i)-\bar y_i)./(1+|y(t_i)|)\|_\infty,$$
being $t_i$, $i=0,\dots,N$, the computational mesh of the considered solver, and $y(t_i)$ and $\bar y_i$ the corresponding values of the solution and of its approximation.

\subsection{Example~1} 
The first problem \cite{Ga2018} is given by:
\begin{eqnarray}\nonumber
y^{(\aa)} &=& -|y|^{1.5} +\frac{8!}{\Gamma(9-\aa) }t^{8-\aa} -
           3\frac{\Gamma(5+\aa/2)}{\Gamma(5-\aa/2)}t^{4-\aa/2} +
           \left( \frac{3}2t^{\aa/2} - t^4 \right)^3  \\[2mm] \label{ex1} &&~+ \frac{9}4\Gamma(\aa+1), \qquad t\in[0,1], \qquad y(0)=0,
\end{eqnarray}
whose solution is
$$ 
y(t) = t^8-3\,t^{4+\aa/2}+\frac{9}4\,t^{\aa}.
$$ 
We consider the value $\aa=0.3$, and use the codes with the following parameters, to derive the corresponding WPD:
\begin{itemize}
\item {\tt flmm2}\,:\, $h=10^{-1}2^{-\nu}$, $\nu=1,\dots,20$; 

\smallskip
\item {\tt fhbvm}\,:\, $M=2,3,4,5$.
\end{itemize}
Figure~\ref{prob2} contains the obtained results: as one may see, {\tt flmm2} reaches less than 12 {\tt mescd}, since by continuing reducing the stepsize, at the 15-th mesh doubling the error starts increasing. On the other hand, the execution time essentially doubles at each new mesh doubling. 
Conversely, {\tt fhbvm} can achieve full machine accuracy by employing a uniform mesh with stepsize $h_1=1/M$, with $M$ very small, thus using very few mesh points. As a result, {\tt fhbvm} requires very short execution times.

\begin{figure}[t]
\centering
\includegraphics[width=9cm]{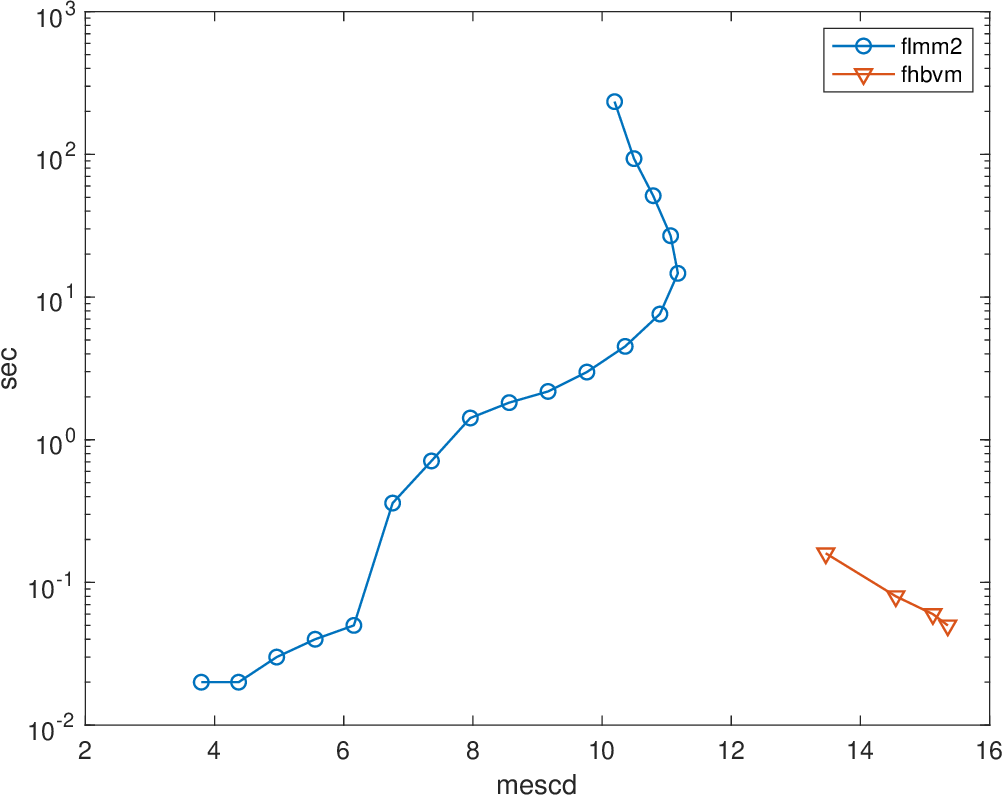}
\caption{Work-precision diagram for problem (\ref{ex1}), $\alpha=0.3$.}
\label{prob2}
\end{figure}

\subsection{Example~2} 
We now consider the following linear problem:
\begin{eqnarray}\nonumber 
y^{(0.5)} &=& \pmatrix{rr}-50 & ~0 \\ -49 &~-1\endpmatrix y, \qquad t\in[0,20],\\[2mm]
y(0)        &=& (2,\, 3)^\top, \label{ex2}
\end{eqnarray}
having solution
$$y(t) = \pmatrix{c} 2\,E_{0.5}\left(-50\cdot t^{0.5}\right)\\[1mm]
2\,E_{0.5}\left(-50\cdot t^{0.5}\right)+E_{0.5}\left(- t^{0.5}\right)\endpmatrix,$$
with $E_{0.5}$ the Mittag-Leffler function.\footnote{We have used the Matlab\cpr function {\tt ml} \cite{Ga2015ml} for its evaluation.} We use the codes with the following parameters, to derive the corresponding WPD:
\begin{itemize}
\item {\tt flmm2}\,:\, $h=10^{-1}2^{-\nu}$, $\nu=1,\dots,20$; 

\smallskip
\item {\tt fhbvm}\,:\, $M=5,\dots,10$.
\end{itemize}
Figure~\ref{sistema} contains the obtained results, from which one deduces that {\tt flmm2} achieves about 5 {\tt mescd} (with a run time of about 85 {\tt sec}), whereas {\tt fhbvm} has approximately 13 {\tt mescd}, with an execution time of about 1 {\tt sec}. Further, in Figure~\ref{error}, we plot the true and estimated (absolute) errors for {\tt fhbvm} in the case $M=10$ (corresponding to a computational mesh made up of 251 mesh-points, with an initial stepsize $h_1\approx 7.3\cdot 10^{-12}$, and a final stepsize $h_{250}\approx 2$): as one may see from the figure, there is a substantial agreement between the two errors.

\begin{figure}[t]
\centering
\includegraphics[width=9cm]{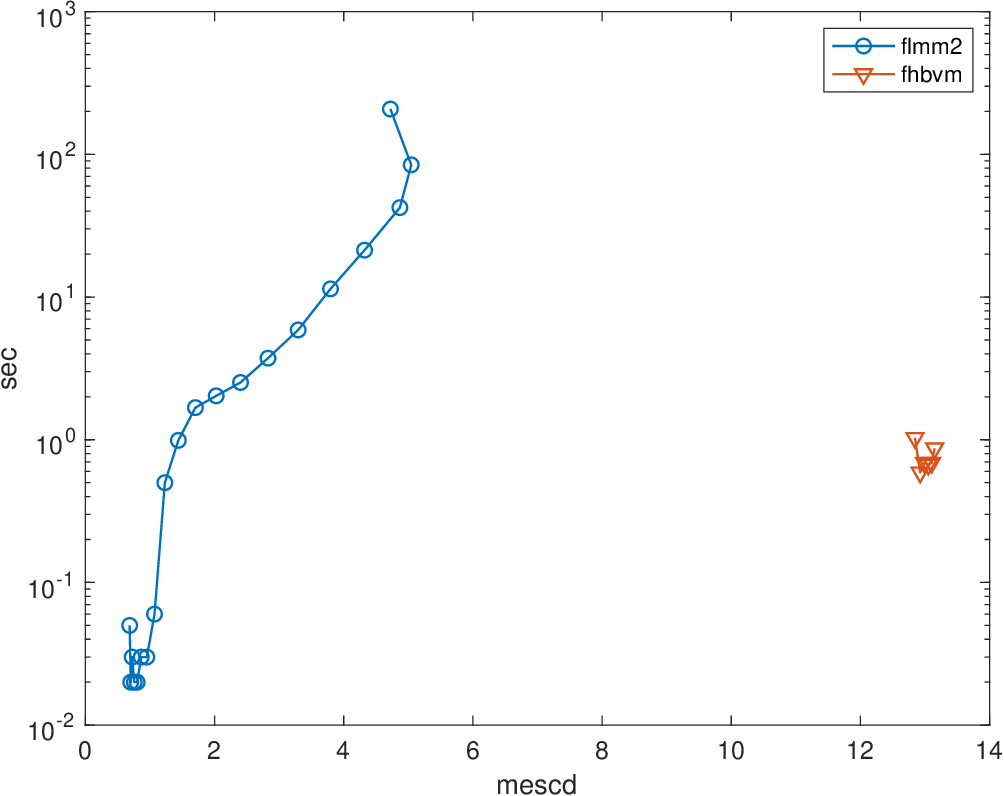}
\caption{Work-precision diagram for problem (\ref{ex2}).}
\label{sistema}

\bigskip
\includegraphics[width=9cm]{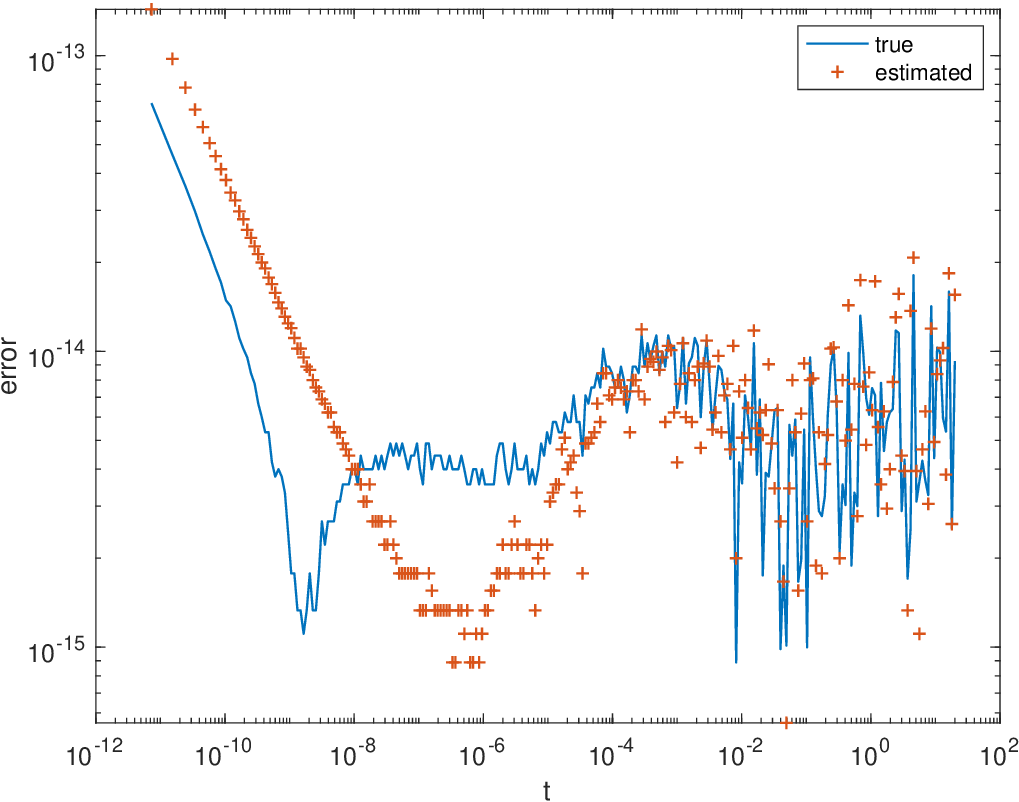}
\caption{True and estimated absolute errors for {\tt fhbvm} solving problem (\ref{ex2}), $M=10$.}
\label{error}
\end{figure}

\subsection{Example~3}  
We now consider the following nonlinear problem \cite{BBBI2023}:
\begin{eqnarray}\nonumber 
y_1^{(1/3)}(t) &=& \frac{t}{10}\left(y_1^3 - \left(\sqrt{y_2}+1\right)^3\right) +  \frac{\Gamma(5/3)}{\Gamma(4/3)}t^{1/3},\\
\nonumber
y_2^{(1/3)}(t)  &=& \frac{1}3\left( y_2^3 - ( y_1-1)^6 \right) + \Gamma(7/3)t,\qquad t\in[0,1],\\[1mm]
y(0)        &=& (1,\, 0)^\top, \label{ex3}
\end{eqnarray}
having solution
$$y(t) = \pmatrix{c} t^{2/3}+1 \\[1mm] t^{4/3}\endpmatrix.$$
This problem is relatively simple and, in fact, both {\tt flmm2} and {\tt fhbvm} solve it accurately. We use it to show the estimated error by using {\tt fhbvm} with parameter $M=2$, which produces a graded mesh with 41 mesh-points,
with $h_1\approx 1.8\cdot 10^{-12}$ and $h_{40}\approx 0.49$. The absolute errors (true and estimated) for each component are depicted in Figure~\ref{sistema1}, showing a perfect agreement for both of them.

In this case, the evaluation of the solution requires  $\approx 0.04$  {\tt sec}, and the error estimation requires   $\approx0.11$ {\tt sec}.

\begin{figure}
\centering
\includegraphics[width=9cm]{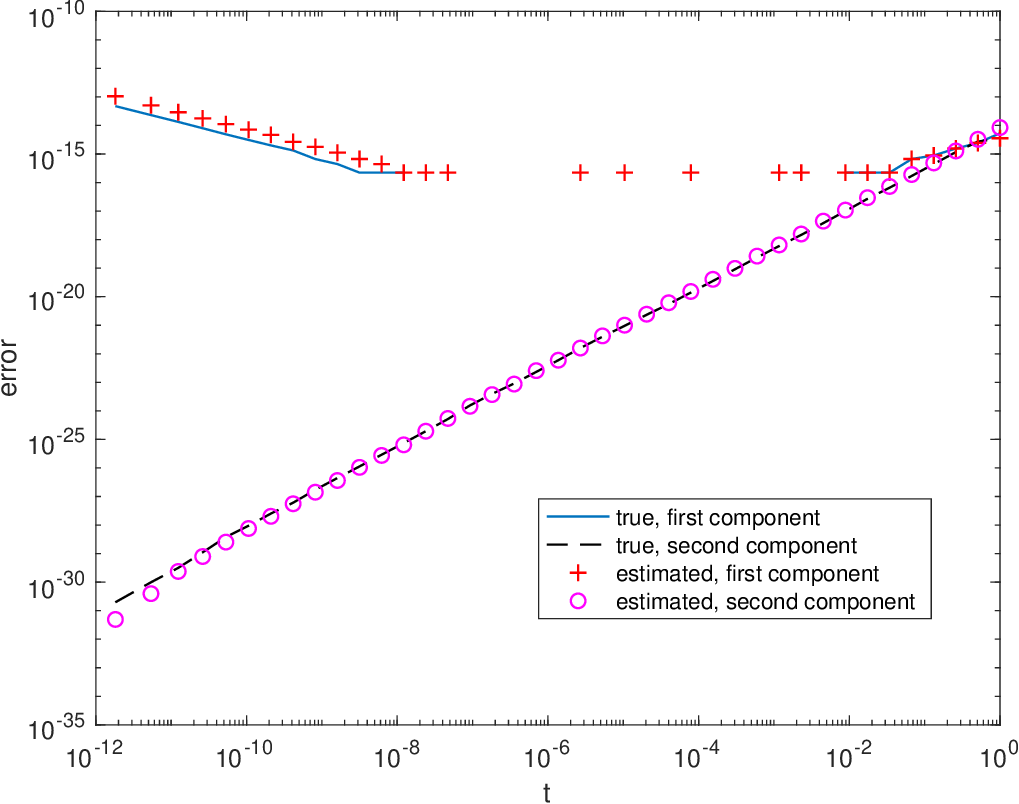}
\caption{True and estimated absolute errors for {\tt fhbvm} solving problem (\ref{ex3}), $M=2$.}
\label{sistema1}
\end{figure}

\subsection{Example~4}  

At last, we consider the following fractional Brusselator model:
\begin{eqnarray}\nonumber 
y_1^{(0.7)} &=& 1-4y_1+y_1^2y_2,\\[2mm] \nonumber
y_2^{(0.7)} &=& 3y_1-y_1^2y_2, \qquad t\in[0,5],\\[2mm]
y(0) &=& (1.2,\,2.8)^\top, \label{ex4}
\end{eqnarray}
By solving this problem using {\tt fhbvm} with parameter $M=5$, a graded mesh of 46 points is produced, with $h_1\approx 6.1\cdot 10^{-5}$ and $h_{45}\approx 0.98$. The maximum estimated error in the computed solution is less than $3.5\cdot 10^{-13}$, whereas the phase-plot of the solution is depicted in Figure~\ref{sistema2}.

In this case, the evaluation of the solution requires  $\approx 0.04$  {\tt sec}, and the error estimation requires   $\approx0.14$ {\tt sec}.

\begin{figure}
\centering
\includegraphics[width=7cm]{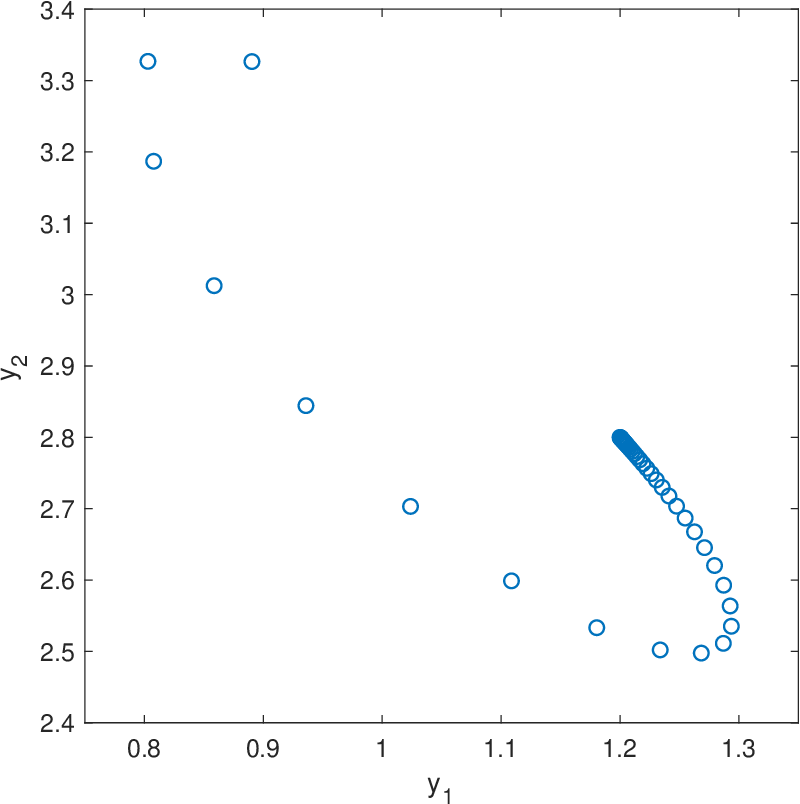}
\caption{Phase-plot of the computed solution by using {\tt fhbvm} solving problem (\ref{ex4}), $M=5$.}
\label{sistema2}
\end{figure}

\section{Conclusions}\label{fine}

In this paper we have described in full details the implementation of the Matlab\cpr\,   code {\tt fhbvm}, able to solving systems of FDE-IVPs.  The code is based on a FHBVM(22,20) method, as described in \cite{BBBI2023}. We have also provided comparisons with another existing Matlab\cpr code, thus confirming its potentialities. In fact, due to the spectral accuracy in time of the FHBVM(22,20) method, the generated computational mesh, which can be either a uniform or a graded one, depending on the problem at hand, requires relatively few mesh points. This, in turn, allows to reduce the execution time due to the evaluation of the memory term required at each step. 

We plan to further develop the code {\tt fhbvm}, in order to provide approximations at prescribed mesh points, as well as to allow selecting different FHBVM$(k,s)$, $k\ge s$, methods \cite{BBBI2023}. At last, we plan to extend the code to cope with values of the fractional derivative, $\aa$, greater than 1.

\paragraph*{\bf Declarations.}\quad The authors declare no conflict of interests, nor competing interests. No grants were received for conducting this study.

\paragraph*{\bf Data availability.}\quad All data reported in the manuscript have been obtained by the Matlab$^\copyright$ code  {\tt fhbvm}, Rel. 2024-03-06, available at the url \cite{fhbvm}.

\paragraph*{\bf Acknowledgements.}\quad The authors are members of the Gruppo Nazionale Calcolo Scientifico-Istituto Nazionale di Alta Matematica (GNCS-INdAM).

\end{document}